\documentclass[12pt]{amsart}
\usepackage{amssymb}
\newtheorem{theorem}{Theorem}
\newtheorem{lemma}[theorem]{Lemma}
\newtheorem{proposition}[theorem]{Proposition}
\newtheorem{corollary}[theorem]{Corollary}

\begin{document}
\title[On cellular-compact and related spaces] {On
cellular-compact and related spaces} 
\author[A. Bella] {Angelo Bella}
\thanks{The research that led to the present paper was partially
supported by a grant of the group GNSAGA of INdAM}
\address{Dipartimento di Matematica e Informatica, viale A. Doria
6, 95125 Catania, Italy}
\email{bella@dmi.unict.it}
\maketitle  
\begin{center}
Dedicated to A. V. Arhangel'ski\u\i\ on the
occasion of his 80th  birthday 
\end{center}

\begin{abstract} {Tkachuk and Wilson proved that a regular
first
countable cellular-compact space has cardinality not exceeding
the continuum. In the same paper they asked if this result
continues to hold for Hausdorff spaces. Xuan and Song considered
the same notion and asked if every cellular-compact space is
weakly Lindel\"of. We answer the last  question  for first
countable
spaces. As a by-product of this result, we present a   somewhat 
different proof of Tkachuk and Wilson theorem, valid  for the
wider class of   Urysohn spaces.  The 
result
actually holds for a class of spaces in between cellular-compact
and cellular-Lindel\"of.   We conclude with some comments on the
cardinality of a weakly linearly Lindel\"of space.} 
\end{abstract} 

\bigskip 
According to \cite{BS}, a space $X$ is cellular-Lindel\" of if
for any  disjoint collection of non-empty  open sets $\mathcal 
U$
there is a Lindel\"of subspace $L$ such that $L\cap U\ne
\emptyset $ for each $U\in \mathcal  U$. This notion has been
further
investigated in \cite{BS1}, \cite{XS1}, \cite{XS2}, \cite{XS3} 
and \cite{Tk}.
Among other 
things, in \cite{BS1} it was shown that  if $2^{<\mathfrak 
c}=\mathfrak 
c$, then the cardinality of a normal first countable 
cellular-Lindel\"of space does not exceed $\mathfrak  c$.  

Recently, Tkachuk
and Wilson have considered the narrower class of cellular-compact
spaces \cite{TW}. A space $X$ is cellular-compact provided that
for any disjoint family $\mathcal  U$ of non-empty open sets
there is
a compact subspace $K$ such that $K\cap U\ne \emptyset $ for each
$U\in \mathcal  U$.In \cite{TW} [Theorem 4.13]  the authors
proved
that the
cardinality of a regular first countable cellular-compact space
does not exceed the continuum and asked [Question 5.1] if this
result could be
true for every Hausdorff space.

In this short note we give a partial answer, by showing that this
happens for the class of Urysohn spaces.  The key point of our
proof is to show that any first countable Hausdorff cellular-
compact space is weakly Lindel\"of with respect to closed sets.
This gives a positive answer to  Question 5.18 in \cite{XS3}
within the class of first countable spaces.   
Then, the cardinality bound valid for Urysohn spaces can be
easily   deduced  from a theorem of Alas \cite{Al}. 
\smallskip 
Recall that a space $X$ is weakly Lindel\"of with respect to
closed sets provided that for any closed set $F$ and any
collection of open sets $\mathcal  U$ such that $F\subseteq
\bigcup \mathcal  U$ there is a countable subcollection $\mathcal
V\subseteq \mathcal 
U$ satisfying $F\subseteq \overline {\bigcup \mathcal  U}$. 
\smallskip 
For notations and undefined notions we refer to \cite{Eng}. A
space $X$ is Lindel\"of (or has the Lindel\"of property) if every
open cover of $X$ has a countable subcover.   
A space is Urysohn if distinct points can be separated by
disjoint closed
neighbourhoods. For a cardinal $\kappa $ and a space $X$ a set
$\{x_\alpha :\alpha <\kappa \}\subseteq X$ is a free sequence if
$\overline
{\{x_\beta:\beta<\alpha \}}\cap \overline {\{x_\beta:\alpha \le
\beta<\kappa \}}= \emptyset $ for each $\alpha <\kappa $. A set
$D\subseteq X$ is strongly discrete if it has a disjoint open
expansion, i. e.  there is a disjoint family of open sets
$\{U_d:d\in D\}$ such that $d\in U_d$ for every $d\in D$. 
\smallskip   
\begin {lemma}\label {Lemma1} If $X$ is a first countable
Hausdorff space,
then  every point has a disjoint local $\pi$-base. \end{lemma}  
\begin{proof} Let $x\in X$. If $x$ is isolated, then there is
nothing to prove.  So, assume it is not isolated  and  fix a
local base
$\{U_n:n<\omega\}$ at $x$. We may assume that $U_{n+1}\subseteq
U_n$ for each $n$.  Since $X$ is a Hausdorff space, there is some
$n_0 $ such that $V_0=U_0 \setminus \overline
{U_{n_0}}\ne \emptyset $. Next, we may choose $n_1>n_0$ such that
$V_1=U_{n_0}\setminus \overline {U_1}\ne\emptyset $,  $n_2>n_1$
such that $V_2=U_{n_1}\setminus \overline {U_{n_2}}\ne\emptyset $
and
so on. It is clear that  the collection $\{V_n:n<\omega\}$ is a
disjoint local $\pi$-base at $x$. \end{proof}
We will say that a space $X$ is strongly cellular-Lindel\"of 
provided that for any disjoint collection $\mathcal  U$ of
non-empty
open sets
there is a closed Lindel\"of subspace $Y$ such that $Y\cap U\ne
\emptyset $ for each $U\in \mathcal  U$. 

It is evident  that  every strongly cellular-Lindel\"of space is
cellular-Lindel\"of and every cellular-compact Hausdorff space is
strongly cellular-Lindel\"of.   The usual $\Psi(\mathcal  A)$
space
over an uncountable almost disjoint family $\mathcal  A$ on
$\omega$
is a cellular-Lindel\"of space which is not strongly cellular-
Lindel\"of. On the other hand,   any countable 
discrete  space is strongly cellular-Lindel\"of but not
cellular-compact.
\begin{lemma} \label {Lemma2}  Let $X$ be a Hausdorff first
countable
strongly cellular-Lindel\"of space. If $D$ is a strongly discrete
subset of $X$, then $\overline D$ has the Lindel\"of property.
\end{lemma} 
\begin{proof} Let  $\{U_d:d\in D\}$ be a
disjoint collection of open sets satisfying $d\in U_d$ for each
$d\in D$. By Lemma \ref{Lemma1}, for  every $d\in D$ we may fix a
disjoint
local $\pi$-base $\mathcal  E_d$ at $d$ such that $\bigcup
\mathcal  E_d\subseteq U_d$. The set $\mathcal 
E=\bigcup\{\mathcal  E_d:d\in
D\}$ is a disjoint collection of non-empty open sets and so there
exists a
closed Lindel\"of subspace $Y$ which intersects each member of
$\mathcal 
E$.   As every member of $\mathcal  E_d$ meets $Y$, we have $d\in
\overline Y=Y$. Therefore, $D\subseteq Y$ and  we deduce that
$\overline D\subseteq Y$ has the Lindel\"of property. \end{proof}
We are now ready to give a partial positive answer to Question
5.18  of  \cite{XS3}. Our proof is inspired by the argument used
in Theorem 4.13 of \cite{TW}. 
\begin{theorem} \label{theor1}  A strongly cellular-Lindel\"of
first countable  Hausdorff space is weakly Lindel\"of with
respect to closed subsets. 
\end{theorem}
\begin{proof} Let $X$ be a strongly cellular-Lindel\"of first
countable Hausdorff space, $F$ a closed subset of $X$ and
$\mathcal 
U$  a collection of open subsets of $X$ satisfying $F\subseteq
\bigcup \mathcal  U$.
Let $x_0\in F$ and take any $W_0=U_0 \in \mathcal  U $ such that
$x_0\in W_0$. We proceed by
induction to define for each $\alpha <\omega_1$ points $x_\alpha
\in F$, open sets $W_\alpha \subseteq U_\alpha  \in \mathcal  U $
with
$x_\alpha \in W_\alpha $ and
countable
families  $\mathcal   V_\alpha\subseteq \mathcal  U$  in such a
way that
the
following conditions are satisfied.

\noindent a) $\overline {\{x_\beta:\beta<\alpha \}}\subseteq
\bigcup \mathcal  V_\alpha $;

\noindent b)  $W_\alpha \cap (\bigcup(\{W_\beta:\beta<\alpha
\}\cup  \bigcup \{\mathcal  V_\beta:\beta<\alpha \}))=\emptyset
$.

Fix $\alpha <\omega_1$ and assume to have already defined 
$\{x_\beta:\beta<\alpha \}$, $\{W_\beta \subseteq U_\beta
:\beta<\alpha \}$   and
$\{\mathcal  V_\beta:\beta<\alpha \}$.  

If $F \subseteq 
\overline {\bigcup (\{W_\beta:\beta<\alpha \}\cup
\bigcup\{\mathcal 
V_\beta:\beta<\alpha \})}$ we stop because   $\mathcal 
V=\{U_\beta:\beta<\alpha \}\cup \bigcup\{\mathcal 
V_\beta:\beta<\alpha \}$ is a countable subfamily of $\mathcal 
U$
satisfying $F\subseteq \overline {\bigcup\mathcal  V}$. If not,  
we may pick a point $x_\alpha \in  F$,  an open  set $W_\alpha$
and an element $U_\alpha  \in \mathcal  U
$ in such a way that $x_\alpha \in W_\alpha \subseteq U_\alpha $
and  
$W_\alpha \cap
(\bigcup(\{W_\beta:\beta<\alpha \}\cup \bigcup\{\mathcal 
V_\beta:\beta<\alpha \}))=\emptyset $. Finally, as the set     
$\{x_\beta:\beta<\alpha \}$ is strongly discrete, by Lemma
\ref{Lemma2}  
the set $\overline {\{x_\beta:\beta<\alpha \}}$ has the
Lindel\"of property.  Since $\overline {\{x_\beta:\beta<\alpha
\}}\subseteq F\subseteq \bigcup\mathcal  U$, there exists a
countable
family $\mathcal  V_\alpha
\subseteq \mathcal  U$ such that $\overline
{\{x_\beta:\beta<\alpha
\}}\subseteq \bigcup\mathcal  V_\alpha $.

Now, at the end of the induction, the resulting set $D=\{x_\alpha
:\alpha\in  \omega_1\}$ turns out to be a free sequence because
for
each $\alpha $ we have $\overline {\{x_\beta:\beta<\alpha
\}}\subseteq \bigcup \mathcal  V_\alpha $ and  $(\bigcup \mathcal
V_\alpha) \cap
\{x_\beta:\alpha \le \beta<\omega_1\}=\emptyset $.

The set $D$ is also strongly discrete and so by Lemma
\ref{Lemma2} its
closure should have the Lindel\"of property. But, a first
countable Lindel\"of space cannot contain  uncountable free
sequences and we reach a contradiction. This shows that the
 induction cannot be carried out for all $\alpha <\omega_1$. As
explained before,  when the induction stops we get a countable
subfamily $\mathcal  V\subseteq \mathcal  U$  satisfying
$F\subseteq
\overline {\bigcup\mathcal  V}$. \end{proof}
The $\Psi$-space over any MAD family of $\omega$ is a
first countable pseudocompact space which is weakly Lindel\"of
with respect to closed sets, but it   is not strongly
cellular-Lindel\"of.  

 Tkachuk constructed in \cite{Tk} a nice  
cellular-Lindel\"of space which is not weakly Lindel\"of.
In a Hausdorff P-space  every Lindel\"of subspace is closed. So,
any   Hausdorff cellular-Lindel\"of P-space is
strongly
cellular-Lindel\"of. Tkachuk's example is a P-space, so we
actually have a strongly cellular-Lindel\"of space which is not
weakly Lindel\"of.

 The
previous space is clearly not first countable. On the other hand,
in
\cite{BS1} it is shown that  under CH  any normal first countable
cellular-Lindel\"of space is weakly Lindel\"of. It is not clear 
whether  there exists a regular (or Hausdorff) first countable
cellular-Lindel\"of space which is not weakly Lindel\"of.

Alas in \cite{Al} proved that  a first countable Urysohn  space
which is weakly Lindel\"of with respect to closed sets has
cardinality at most the continuum. So, we immediately get:
\begin{corollary} If $X$ is a first countable Urysohn
strongly cellular-Lindel\"of 
( in particular cellular-compact) space, then $|X|\le \mathfrak 
c$.
\end{corollary}
It is still an open problem  whether Alas's result is true for
Hausdorff spaces.  In \cite{Ar}[Corollary 22] Arhangel'ski\u\i\
showed 
that this  happens by
using the stronger notion of strict  quasi Lindel\"ofness. Such
result makes sense only for first countable spaces, but  a more
general one  has been recently established in \cite{BS2}. A space
$X$
is strictly quasi Lindel\"of provided that for any closed set
$F$ and any collection of open sets $\mathcal 
U=\bigcup\{\mathcal 
U_n:n<\omega\}$ such that $F\subseteq \bigcup \mathcal  U$ there
are
countable subcollections $\mathcal  V_n\subseteq \mathcal  U_n$
for each
$n<\omega$ satisfying $F\subseteq \bigcup \{\overline
{\bigcup\mathcal  V_n}: n<\omega\}$. Unfortunately,  
here we did not manage  to  prove that a first countable
Hausdorff cellular-compact space is strictly quasi Lindel\"of.

Another partial answer to the question  asked by Tkachuk and
Wilson can be obtained by strengthening  the hypothesis of first
countability.
\begin{theorem} \label{theor2} If $X$ is a cellular-compact space
with a point-countable base, then $|X|\le \mathfrak c$.
\end{theorem}
\begin{proof} Since $X$ has a point-countable base, every compact
subset is metrizable and hence separable.  This in turn implies
that every collection of pairwise disjoint open sets is
countable. Now, the result follows  from the Hajnal-Juh\'asz
inequality $|X|\le 2^{c(X)\chi(X)}$. \end{proof}

As  a  Hausdorff cellular-Lindel\"of P-space is
strongly
cellular-Lindel\"of, with minor modifications we may prove:
\begin{theorem} A cellular-Lindel\"of Urysohn P-space of
character  at most  $\omega_1$ has cardinality not exceeding
$2^{\omega_1}$. \end{theorem}

\smallskip 
According to \cite{jtw}, a space  is weakly linearly Lindel\"of
if
every family of non-empty open sets has a complete accumulation
point. A collection of sets $\mathcal  U$ in the space $X$ has an
accumulation point $p$ if every neighbourhood of $p$ meets
$|\mathcal 
U|$-many elements of $\mathcal  U$.

A space $X$ is almost linearly Lindel\"of \cite{Tk} if every open
cover has a subcollection of countable cofinality whose union is
dense in $X$

Corollary 3.3 in \cite{Tk} shows that every almost linearly
Lindel\"of space is weakly linearly Lindel\"of  and Corollary
3.19 in \cite{Tk}  exhibits under CH (but the argument actually
works by assuming $\mathfrak  c<\aleph_\omega$) a
cellular-Lindel\"of
space
which is not almost linearly Lindel\"of.
\begin {lemma} \label{Lemma6}
 Let $X$ be a   Hausdorff sequential space. If
$A\subseteq X$, then $|\overline A|\le |A|^\omega$. \end{lemma}

\begin{lemma} \label{Lemma7} \cite{Tk} [Theorem 3.11]   Let $X$
be a 
weakly linearly Lindel\" of
space. If
$\mathcal   W$ is an open cover of $X$ of regular uncountable
cardinality,  then there exists a
subfamily
$\mathcal   W'\subseteq \mathcal   W$ such that $\bigcup \mathcal
W'$ is dense in
$X$ and $|\mathcal   W'|< |\mathcal W| $. \end{lemma} 
\begin{theorem} \label{t2}  Let  $X$
 be a
normal sequential space satisfying
$\chi(X)\le
\mathfrak   c$. If either 
a) [$2^{<\mathfrak  c}=\mathfrak  c$]\  $X$ is weakly
linearly
Lindel\"of 
or b) [$\mathfrak  c<\aleph_\omega$]\ $X$ is almost
linearly
Lindel\"of,
then $|X|\le \mathfrak  c$. \end{theorem} 
\begin{proof}
 For  each $p\in X$ let $\mathcal    U_p$ be a
local base at $p$ such that $|\mathcal    U_p|\le \mathfrak   
c$. We
will
construct by transfinite recursion an increasing 
sequence $\{H_\alpha :\alpha <\mathfrak    c\}$ of
closed subsets of $X$ satisfying:

\noindent  $1_ \alpha $) $|H_\alpha |\le \mathfrak    c$;

\noindent $2_{\alpha} $) if $X\setminus \overline {\bigcup
\mathcal   V}\ne \emptyset$  for some $\mathcal    V\subseteq 
\bigcup\{\mathcal  
U_p:p\in
H_\alpha \}$ with $|\mathcal  V|{<\mathfrak    c}$ (case a) or
$|\mathcal  V|\le
\omega$ (case b)), then $H_{\alpha +1}\setminus
\overline
{\bigcup\mathcal    V}\ne \emptyset $.

Put $H_0=\{x_0\}$ for some $x_0\in X$ and let $\phi :\mathcal  
P(X)\rightarrow X$ be any choice function    such that
$\phi(\emptyset)=x_0$.
Assume to have already defined  the subsequence
$\{H_\beta:\beta<\alpha \}$. If $ \alpha $ is a limit ordinal,
then put $H_\alpha =\overline{\bigcup\{H_\beta:\beta<\alpha \}}$
(Lemma \ref{Lemma6} ensures $1_\alpha $).
If $\alpha =\gamma+1$, then  let $H_\alpha$ be the closure of the
set $
H_\gamma\cup\{\phi(X\setminus \overline{\bigcup \mathcal  
V}):\mathcal  
V\subseteq \bigcup\{\mathcal     U_p :p\in H_\gamma\}$ with
$|\mathcal  V| 
{<\mathfrak    c}$ (case a)) or $|\mathcal  V|\le \omega$ (case
b)) 
$\}$. A  counting argument (which takes into account
$2^{<\mathfrak 
c}=\mathfrak  c$ in case a))   and again Lemma \ref{Lemma6} 
show that
$H_\alpha $ satisfies $1_\alpha $.

 Then, put $H=\bigcup\{H_\alpha :\alpha <\mathfrak    c
\}$. It is clear that
$|H|\le \mathfrak    c
$. So, the proof will be completed by showing that
$X=H$.  Suppose the contrary and pick a  non-empty open set $O$
such that $\overline O\subseteq X\setminus H$.  For each $p\in H$
take an element $U_p\in \mathcal    U_p$  satisfying $U_p\cap
O=\emptyset $. As the space is normal, we may also pick an open
set $W$ such that $H\subseteq W$ and $\overline W\subseteq
\bigcup\{U_p: p\in H\}$. The collection $\mathcal    W=\{U_p:p\in
H\}\cup \{X\setminus \overline W\}$ is an open cover of $X$ of
cardinality not exceeding $\mathfrak    c$.

\noindent Case a). If $|\mathcal  W|=\mathfrak  c$, then  by
Lemma  \ref{Lemma7} 
there
exists a
subfamily $\mathcal    W'\subseteq \mathcal    W$ such that
$\bigcup
\mathcal    W'$
is dense in $X$ and $|\mathcal    W'|<  \mathfrak  
c$. If $|\mathcal  W|<\mathfrak  c$, then  let $\mathcal 
W'=\mathcal  W$.

\noindent Case b).  By  the definition of almost linear
Lindel\"ofness, 
there
exists a
subfamily $\mathcal    W'\subseteq \mathcal    W$ of countable
cofinality
such that $\bigcup
\mathcal    W'$
is dense in $X$. Since  $|\mathcal    W'|\le |\mathcal    W|\le
\mathfrak  
c\le \aleph_\omega$, we must have  $|\mathcal  W'|\le \omega$.

In both cases,    put
$\mathcal  
V=\mathcal    W'\setminus \{X\setminus \overline W\}$ and notice
that
$H\subseteq \overline {\bigcup \mathcal    V}$.

Moreover, in both cases we have $|\mathcal  V|<cf(\mathfrak  c)$.
Thus,  
there is an ordinal $\alpha <\mathfrak    c$ such that
$\mathcal    V\subseteq \{U_p:p\in H_\alpha \}$. Since $H_{\alpha
+1}\subseteq H\subseteq \overline {\bigcup\mathcal    V}$, we
reach
a
contradiction with condition $2_ \alpha $. This completes the
proof. \end{proof}
As mentioned in \cite{Tk}, it is still an open  problem to prove
Theorem \ref{t2}  in ZFC. As in Theorem \ref{theor2},  
we observe that a positive solution can
be
easily obtained by strengthening  the first countability
assumption.

\begin {theorem} If $X$ is a normal weakly linearly
Lindel\"of space with a point-countable base, then $|X|\le
\mathfrak 
c$. \end{theorem} 
\begin{proof} By Proposition 3.10 in \cite{Tk},  the space $X$
has extent $e(X)\le \mathfrak  c$. Now, it suffices to use the
inequality $|X|\le we(X)^{psw(X)}$, established by Hodel in
\cite{Ho} (see also \cite{BB} for a short direct proof).  Since
in our case $we(X)\le e(X)\le \mathfrak  c$ and $psw(X)=\omega$,
we
are done. \end{proof}

Recall that the weak extent $we(X)$ of a space $X$ is the
smallest cardinal $\kappa $  such that for any open cover
$\mathcal U$ there is a set $A\subseteq X$ satisfying $|A|\le
\kappa $ and $X=\bigcup\{U: U\in \mathcal U, U\cap A\ne \emptyset
\}$.

We conclude this paper with some comments on a recent work of
Xuan and Song  \cite{cinesi}.

Denote by $\tau(X)$ the topology of the space $X$.
A g-function for  $X$ is a map $g:\omega \times X\to
\tau(X)$ such that $x\in g(n,x)$ and $g(n+1,x)\subseteq g(n,x)$
for each $x\in X$ and $n<\omega$. A g-function $g:\omega\times
X\to \tau(X)$ is said to be symmetric if for any $n<\omega$ and
$x,y\in X$ $x\in g(n,y)$ if and only if $y\in g(n,x)$. 
Furthermore,  $g^2(n,x)=\bigcup\{g(n,y):y\in g(n,x)\}$.

In \cite{cinesi} the authors  gave several results on weakly
linearly Lindel\"of spaces. We wish to make a comment on one of
them.     

\begin{proposition} \label {o}  \cite{cinesi} [Theorem 3.14]
Suppose $X$ is a
Baire space with a symmetric function  $g$  such that:

\begin{enumerate}
\item $\bigcap\{g^2(n,x):n<\omega\}=\{x\}$ for each $x\in
X$;
\item for each $n<\omega$ there is a set $F_n\subseteq X$
such that $|F_n|\le \omega_1$ and $X=\bigcup\{g(n,x):x\in F_n\}$.
\end{enumerate}

If every family of non-empty open subsets of $X$ of cardinality
$\omega_1$ has a complete accumulation point, then $|X|\le
\mathfrak c$. \end{proposition}

We observe that in the above result neither the Baire property
nor that fragment of  weak linear Lindel\"ofness are needed.

\begin{theorem} 
Suppose that $X$ is a space with a symmetric  g-function 
$g:\omega\times X\to \tau(X)$ satisfying:

\begin{enumerate}
\item $\bigcap \{g^2(n,x):n<\omega\}=\{x\}$ for each $x\in
X$; 
\item for each $n<\omega$ there is a set  $A_n\subseteq
X$ such that $X=\bigcup\{g(n,a):a\in A_n\}$ and $|A_n|\le
\mathfrak c$.  
\end{enumerate}

Then $|X|\le \mathfrak c$.
\end{theorem}

\begin{proof}
 Let $A=\bigcup\{A_n:n<\omega\} $. Note that $|A|\leq \mathfrak
c$. 
Fix  a well-ordering $\prec$ on $X$.
We may define a map $f:X\to {}^\omega A $ in such a way that for
$x\in X$ and  $n<\omega$ we have that $f(x)(n) = a$, where $a$ is
the $\prec$-first  element in $
A_n$ satisfying $x\in g(n,a)$. To complete the
proof we will show that this mapping is injective.

So fix $x\not= y$. Then we may find  $n<\omega$   such that
$ y\notin g^2(n,x)$. Since we are assuming that $g$ is symmetric,
the latter formula is equivalent to $g(n,x)\cap g(n,y)=\emptyset
$. 
Now let $p=f(x)(n)$. Then $x\in g(n,p) $, and then also
$p\in g(n,x) $.   But, this means that
$p\notin g(n,y)$  and therefore
$y\notin g(n,p)$. This  implies that
$p\not=f(y)(n)$. This shows that $f$ is injective and
we are done. \end{proof}

Recall that a space $X$ has a $G_\delta$-diagonal of rank 2
provided that there is a sequence of open covers $\{\mathcal
U_n:n<\omega\}$ satisfying $\bigcap St^2(x,\mathcal U_n)=\{x\}$
for each $x\in  X$. From Proposition \ref{o} the authors of
\cite{cinesi} derived the following:     
\begin {corollary} \cite{cinesi} [Corollary 3.15] 
If $X$ is a weakly linearly Lindel\"of Baire space with a
$G_\delta$-diagonal of rank 2 such that $we(X)\le \omega_1$,  
then $|X|\le \mathfrak c$. \end{corollary}
As before, this result has a better formulation, already
established in \cite{bbr}.
\begin{corollary} \cite{bbr} [Proposition 4.3]  If $X$ is a space
with a $G_\delta$-diagonal of rank 2 and $we(X)\le \mathfrak c$,
then $|X|\le \mathfrak c$. \end{corollary}

 Proposition 4.5  in \cite{bbr}  establishes  that  the
cardinality of a weakly Lindel\"of Baire space with a $G_\delta$-
diagonal of rank 2  does not exceed $\mathfrak c$, while it is
still an open problem if a similar result continues to hold
without the Baire assumption. In this direction, Questions 5.1
and
5.2 in \cite{cinesi} appear very interesting. Indeed,   these
questions ask if the cardinality of a weakly
linearly Lindel\"of  (Baire) space with a $G_\delta$-diagonal of
rank 2 is bounded by $\mathfrak c$.  This happens for sure within
the
class of normal spaces.
\begin{theorem} If $X$ is a weakly linearly Lindel\"of normal
space with a $G_\delta$-diagonal of rank 2, then $|X|\le
\mathfrak c$. \end{theorem}
\begin{proof} To obtain  a contradiction, assume  that
$|X|>\mathfrak c$ and
let $\{\mathcal U_n:n<\omega\}$ be a sequence of open covers
witnessing a $G_\delta$-diagonal of rank 2. For each $n<\omega$
let $F_n=\{\{x,y\}:x,y\in X, y\notin St^2(x,\mathcal U_n)\}$. We
have $[X]^2=\bigcup\{F_n:n<\omega\}$ and so by Erdos-Rado theorem
there is some $n_0\in \omega$ and an uncountable set $H\subseteq
X$  such that $[H]^2\subseteq F_{n_0}$. The collection
$\{St(x,\mathcal U_{n_0}):x\in H\}$ consists of pairwise
disjoint sets and $H$ is closed in $X$. By the normality of $X$
there is an open set $V$ satisfying $H\subseteq V$ and $\overline
V\subseteq \bigcup \{St(x,\mathcal U_{n_0}): x\in H\}$. But then,
the uncountable family of open sets $\{V\cap St(x,\mathcal
U_{n_0}):x\in H\}$ has no complete accumulation point, 
contradicting the weak linear Lindel\"ofness of $X$. \end{proof}

\end{document}